\UseRawInputEncoding
\documentclass[10pt]{article}

\usepackage{amsfonts,amsmath,amsthm}
\usepackage{amssymb}
\usepackage{algorithm}
\usepackage{algpseudocode}
\usepackage{xcolor}
\usepackage{graphicx}
\usepackage{stmaryrd}        % provides \llbracket and \rrbracket
\usepackage{multirow}
\usepackage{dsfont}
\usepackage{rotating}
\usepackage{booktabs}

\usepackage[paper=a4paper,dvips,top=2cm,left=2cm,right=2cm,
foot=1cm,bottom=4cm]{geometry}

\newcommand{\vx}{{\bf x}}
\newcommand{\vy}{{\bf y}}

\newcommand{\R}{\mathbb{R}}

\newcommand{\sos}{\mathrm{sos}}

\newtheorem{Thm}{Theorem}[section]
\newtheorem{Def}[Thm]{Definition}

\newtheorem{example}[Thm]{Example}
\newtheorem{remark}[Thm]{Remark}

\begin{document}
	
\title{Biquadratic SOS Rank and Augmented Zarankiewicz Number}
	\Large
	\author{Liqun Qi\footnote{Jiangsu Provincial Scientific Research Center of Applied Mathematics, Nanjing 211189, China.
			Department of Applied Mathematics, The Hong Kong Polytechnic University, Hung Hom, Kowloon, Hong Kong.
			({\tt maqilq@polyu.edu.hk})}
		\and
		Chunfeng Cui\footnote{School of Mathematical Sciences, Beihang University, Beijing  100191, China.
			({\tt chunfengcui@buaa.edu.cn})}
		\and {and \
			Yi Xu\footnote{School of Mathematics, Southeast University, Nanjing  211189, China. Nanjing Center for Applied Mathematics, Nanjing 211135,  China. Jiangsu Provincial Scientific Research Center of Applied Mathematics, Nanjing 211189, China. ({\tt yi.xu1983@hotmail.com})}
		}
	}

	\date{\today}
	\maketitle

\begin{abstract}
This paper introduces the concepts of the augmented Zarankiewicz number \(z_A(m,n)\) and the limited augmented Zarankiewicz number \(z_L(m,n)\), which are natural combinatorial extensions of the classical Zarankiewicz number. These numbers arise from augmented bipartite graphs that may contain both standard edges (1-edges) and pairs of edges representing squares of binomials (2-edges). The main theoretical result establishes the inequality chain \(\operatorname{BSR}(m, n) \geq z_A(m, n) \geq z_L(m, n) \geq z(m, n)\), linking the maximum biquadratic sum-of-squares (SOS) rank to these extremal graph parameters. We determine the exact values of \(z_L(m, n)\) for the cases \((m,2)\), \((3,3)\), $(4, 3)$ and \((4,4)\), and provide new lower bounds for the cases \((5,3)\), \((5,4)\), and \((5,5)\). These results yield improved lower bounds for the maximum SOS rank of biquadratic forms, demonstrating that \(z_L(m,n)\) can exceed the classical Zarankiewicz number, thereby offering a refined combinatorial perspective on the SOS rank problem.

\medskip

\noindent\textbf{Keywords.} Biquadratic form, sum of squares, SOS rank, Zarankiewicz number, augmented Zarankiewicz number, bipartite graph, limited augmented Zarankiewicz number, \(C_4\)-cycle.

\medskip
\noindent\textbf{AMS subject classifications.} 14P10, 05C35, 11E25, 15A69, 90C22.
\end{abstract}

\section{Introduction}

Denote $\{1, \ldots, m\}$ as $[m]$.   Assume that $m \ge  n \ge 2$.   Let
\begin{equation} \label{e1}
P(\mathbf{x}, \mathbf{y}) = \sum_{i,k=1}^{m} \sum_{j,l=1}^{n} a_{ijkl} x_i x_k y_j y_l.
\end{equation}
We call $P$ an $m \times n$ {\bf biquadratic form}.    Here $a_{ijkl}$ are real numbers.  We assume that
\begin{equation} \label{e2}
a_{ijkl} = a_{kjil} = a_{klij}
\end{equation}
for $i, k \in [m], j, l \in [n]$.
A PSD (positive semi-definite) biquadratic form is one for which $P(\mathbf{x}, \mathbf{y}) \geq 0$ for all $\mathbf{x}, \mathbf{y}$. It is SOS (sum of squares) if it can be written as a finite sum of squares of bilinear forms, i.e.,
\begin{equation} \label{e3}
P(\vx, \vy) = \sum_{p=1}^r f_p(\vx, \vy)^2,
\end{equation}
where $f_p$ are $m \times n$ bilinear forms for $p \in [r]$.  The smallest $r$ such that (\ref{e3}) holds is called the SOS rank of $P$, denoted as $\mathrm{sos}(P)$.  In this case, we say that the sum of squares expression $\sum_{p=1}^r f_p(\vx, \vy)^2$ is {\bf irreducible}.  The maximum SOS rank of $m \times n$ SOS biquadratic forms is denoted as $BSR(m, n)$.   It is known that $BSR(m, 2) = m+1$ and $BSR(2, n) = n+1$ \cite{BPSV19}, \(BSR(3,3) = 6\) \cite{BSSV22} and $BSR(m, n) \le mn-1$ \cite{QCX26}.

Recently, it was proved \cite{CQX26} that
\begin{equation} \label{e4}
BSR(m ,n) \ge z(m, n),
\end{equation}
where $z(m, n)$ is the Zarankiewicz number \cite{Za51}.  The Zarankiewicz problem, originating from \cite{Za51}, asks for the maximum number of edges in an $m \times n$ bipartite graph that contains no complete bipartite subgraph $K_{2,2}$ (i.e., no $C_4$).  This classical extremal problem has been extensively studied; see \cite{Gu69, Bo04, RS65} for surveys and fundamental bounds.

The inequality (\ref{e4}) is based upon the simple biquadratic forms introduced in \cite{QCX26}.  A biquadratic form is a simple biquadratic form if it contains only distinct terms of the type \(x_i^2 y_j^2\).   Then we may establish a one-to-one relation.  Let $G = (S, T, E)$ be a bipartite graph, where $S = [m]$ and $T = [n]$ are the vertex sets and $E$ is the edge set. Then uniquely we have a simple quadratic form
\begin{equation} \label{e5}
P_G(\vx, \vy) = \sum_{(i, j) \in E} x_i^2y_j^2.
\end{equation}
If $G$ has no $C_4$-cycles, then the simple biquadratic form $P_G$ is irreducible, i.e., $\mathrm{sos}(P_G) = |E|$.  On the other hand, the Zarankiewicz number $z(m, n)$ is defined as the largest number of edges such that the $m \times n$ bipartite graph $G$ is $C_4$-cycle free \cite{Za51, RS65}.   Thus, we have (\ref{e4}).

It is known that $BSR(m, 2) = z(m , 2) = m+1$, $BSR(2, n) = z(2, n) = n+1$ and $BSR(3, 3) = z(3, 3) = 6$.  However, it was found in \cite{XCQ26} that $BSR(4, 3) \ge 8 > z(4, 3) = 7$.  The value $z(4,3)=7$ is a classical result in Zarankiewicz theory \cite{RS65, Gu69}.  This gap was discovered in the following example.
% way.   First, we have
%\[
%P_{4,3,7}(\vx, \vy) = x_1^2 y_1^2 + x_2^2 y_2^2 + x_3^2 y_3^2 + x_1^2 y_2^2 + x_2^2 y_3^2 + x_3^2 y_1^2 + x_4^2 y_1^2.
%\]
%This is a $4 \times 3$ simple biquadratic form whose corresponding $4\times 3$ bipartite graph $G$ attains the Zarankiewicz number $z(4, 3)=7.$   Then, an $8$-square $4 \times 3$ sos biquadratic form $Q$ is formed as
%\begin{equation}\label{equ:Q437}
%	Q(\vx, \vy) = P_{4, 3, 7}(\vx, \vy) + (x_4y_2 + x_1y_3)^2.
%\end{equation}
%It was proved that $Q$ is irreducible, i.e., $\mathrm{sos}(Q) = 8$.  In this way, it was established that
%\begin{equation} \label{e6}
%BSR(4, 3) \ge 8 > z(4, 3) = 7.
%\end{equation}
\begin{example}\cite{XCQ26}\label{example:BSR43}
	Define an $8$-square $4 \times 3$ sos biquadratic form as
	\begin{equation}\label{equ:Q437}
		Q(\vx, \vy) = P_{4, 3, 7}(\vx, \vy) + (x_4y_2 + x_1y_3)^2,
	\end{equation}
	where \[
	P_{4,3,7}(\vx, \vy) = x_1^2 y_1^2 + x_2^2 y_2^2 + x_3^2 y_3^2 + x_1^2 y_2^2 + x_2^2 y_3^2 + x_3^2 y_1^2 + x_4^2 y_1^2.
	\]
\end{example}
	
It was proved that $Q$ is irreducible, i.e., $\mathrm{sos}(Q) = 8$.  In this way, it was established that
		\begin{equation} \label{e6}
			BSR(4, 3) \ge 8 > z(4, 3) = 7.
		\end{equation}

This makes us think further: Why can such a thing happen?  What is its meaning in the theory of the Zarankiewicz number?  Now,
$(x_4y_2 + x_1y_3)^2$ is not a single term square.  It does not correspond to an edge of the bipartite graph $G$.   But can we interpret it as a super edge of $G$?  This motivates this paper.

The remainder of this paper is organized as follows. Section~2 formally defines the augmented bipartite graph, the augmented Zarankiewicz number \(z_A(m,n)\), and the limited augmented Zarankiewicz number \(z_L(m,n)\), establishing the fundamental relationship \(\operatorname{BSR}(m, n) \geq z_A(m, n) \geq z_L(m, n) \geq z(m, n)\). Sections~3 and 4 determines the exact values of \(z_L(m, n)\) for the cases \((m,2)\), \((3,3)\), $(4, 3)$ and \((4,4)\). Sections~5, 6, and~7 investigate and establish new lower bounds for \(z_L(m,n)\) for the cases \((5,3)\), \((5,4)\), and \((5,5)\), respectively. Finally, Section~8 concludes the paper by summarizing the results and presenting several open problems for future research.

\section{Augmented Zarankiewicz Number and Limited Augmented Zarankiewicz Number}

Let $G_1 = (S, T, E_1)$ be an $m \times n$ bipartite graph, where $S = [m]$ and $T = [n]$ are its vertex sets.  Assume that $G_1$ has no $C_4$-cycles.  Then we say that $G_1$ can be augmented to an {\bf augmented bipartite graph} $G = (S, T, E)$, where the edge set $E = E_1 \cup E_2$.  Here, we call any edge $(i, j)$ of $G_1$ a {\bf 1-edge} of $G$, while $E_2$ is the {\bf 2-edge} set of $G$.

A $1$-edge $e$ in $E_1$ has the form $e = (i, j)$, where $i \in S$ and $j \in T$. On the other hand, a $2$-edge $e$ in $E_2$ has the form $(i, j; k, l)$, where $i, k \in S$ and $j, l \in T$.

\begin{figure}
		\begin{center}
			\includegraphics[width=0.8\linewidth]{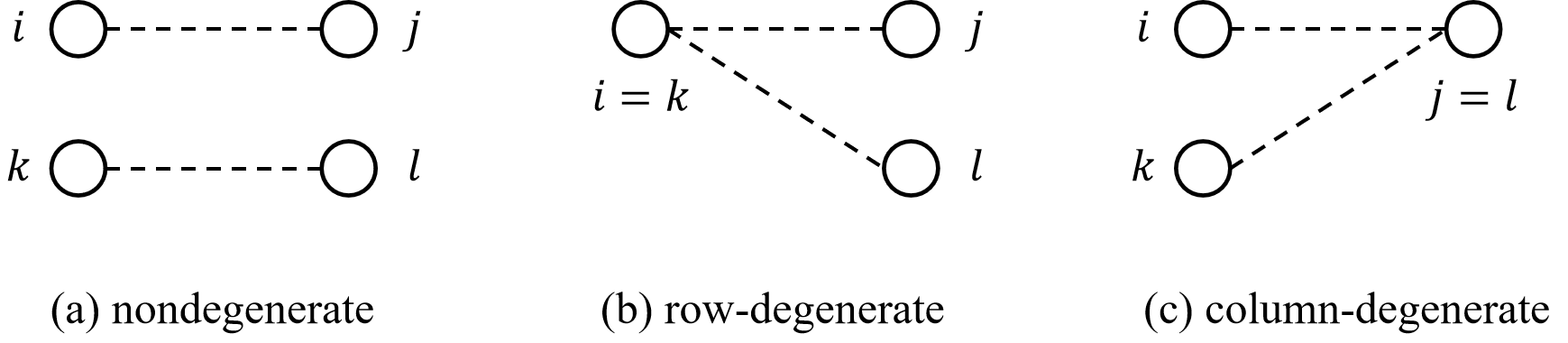}
		\end{center}	
		\caption{Illustrations of nondegenerate, row-degenerate, and column-degenerate $2$-edges.}\label{fig:2edge}
\end{figure}

A $2$-edge can be \emph{nondegenerate} if $i \neq k$ and $j \neq l$, \emph{row-degenerate} if {$i=k$ and $j\neq l$,} %$i \neq k$ and $j = l$,
 or \emph{column-degenerate} if {$i \neq k$ and $j = l$.}  %$i = k$ and $j \neq l$.
  It cannot be $i = k$ and $j = l$.
{Illustrations of nondegenerate, row-degenerate, and column-degenerate $2$-edges are shown in Fig.~\ref{fig:2edge}.}
Furthermore, we impose the following \emph{simplicity condition} on such an augmented bipartite graph $G$:
\begin{enumerate}
\item[(S)] No 2-edge overlaps with any 1-edge or other 2-edge on a cell.
\end{enumerate}
Here, we call $(i, j)$ a \textbf{cell} for any $i \in S$ and $j \in T$.
If $|E_1| = z(m, n)$, we say that $G$ is a {\bf limited augmented bipartite graph}.

For such an augmented bipartite graph $G$, we associate it with an $m \times n$ sos biquadratic form $P_G$, defined as:
\begin{equation} \label{e7}
P_G(\vx, \vy) = \sum_{(i, j)\in E_1} x_i^2y_j^2 + \sum_{(i, j; k, l) \in E_2} (x_iy_j+x_ky_l)^2.
\end{equation}
We call $P_G$ a \textbf{doubly simple biquadratic form}. If $P_G$ is irreducible, then the SOS rank of $P_G$ is $|E| = |E_1| + |E_2|$.
 For instance, Example~\ref{example:BSR43} is a $4\times 3$ irreducible doubly simple biquadratic form  with
 $$E_1 = \{ (1, 1), (2, 2), (3, 3), (1, 2), (2, 3), (3, 1), (4, 1)\} \text{ and } E_2 = \{ (4,2;1,3) \}.$$ % $E_2 = \{ (i,j;k,l) \}$.
  %Then we see that if $(i, j;k,l)$ is $(4,2;1,3)$,
 % $P_G$ is irreducible.
 We wish to characterize the combinatorial feature of irreducible doubly simple biquadratic forms.
  %such
  %$(i, j; k, l)$.

\medskip

We now refine the notion of a $C_4$-cycle to account for the presence of 2-edges.

\begin{Def}[Generalized $C_4$-Cycle, Augmented Zarankiewicz Number and Limited Augmented Zarankiewicz Number] \label{d2.1}
Let $G = (S,T,E_1 \cup E_2)$ be an $m \times n$ augmented bipartite graph with vertex sets $S=[m]$, $T=[n]$, augmented from a $C_4$-cycle free bipartite graph $G_1 = (S,T,E_1)$.
A cell $(i,j)$ is called \textbf{occupied} if $(i,j) \in E_1$ or $(i,j)$ is a half of some $2$-edge in $E_2$.
We say that $G$ contains a \textbf{generalized $C_4$-cycle} if either
\begin{enumerate}
\item there exists a classical $C_4$-cycle formed by $1$-edges, or
\item there exists a nondegenerate $2$-edge $(i,j;k,l)\in E_2$ such that both opposite cells $(i,l)$ and $(k,j)$ are occupied, or
\item there exist a $2$-edge $(i, j; p, q)$ (of any type) and a distinct cell $(k, l)$, with $k \not = i$, $k \not = p$, $l \not = j$, $l \not = q$, such that the five cells $(k, l), (k, j), (k, q), (i, l), (p, l)$ are all occupied. Furthermore, if the 2-edge is nondegenerate, these five cells must be pairwise distinct.
\end{enumerate}

The \textbf{augmented Zarankiewicz number} $z_A(m,n)$ is the maximum possible total number of edges $|E_1|+|E_2|$ for which a generalized $C_4$-cycle does not exist for such an augmented bipartite graph $G$.

The \textbf{limited augmented Zarankiewicz number} $z_L(m,n)$ is the maximum possible total number of edges $|E_1|+|E_2|$ for which a generalized $C_4$-cycle does not exist for a limited augmented bipartite graph $G$.
\end{Def}

	\begin{figure}
		\begin{center}
			\includegraphics[width=\linewidth]{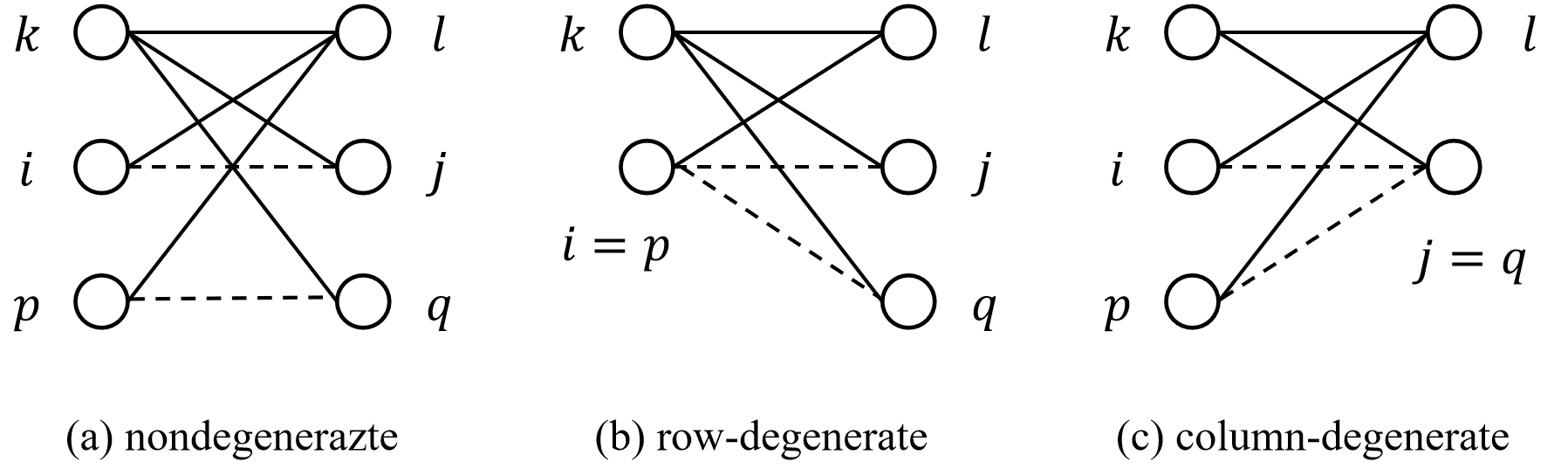}
		\end{center}	
		\caption{An example that Condition 3 in Definition~\ref{d2.1} is triggered in the nondegenerate, row-degenerate, and column-degenerate cases. Here, each solid line represents 1-edge, and the two dotted lines  represent  a 2-edge.}\label{fig:condition3}
\end{figure}

\begin{remark}\label{rem:2.2}
Conditions 1 and 2 are natural and combinatorial. Condition 1 ensures that no classical $C_4$ appears. Condition 2 is necessary: if a nondegenerate $2$-edge $(i,j;k,l)$ had both opposite cells $(i,l)$ and $(k,j)$ occupied, then the associated vectors {may} satisfy a linear dependence, lowering the SOS rank.
For a limited augmented bipartite graph, because $|E_1| = z(m,n)$, each half of a $2$-edge would form a $C_4$ with some distinct occupied cells (otherwise more $1$-edges could be added). Condition 3 then guarantees that the two lists of such cells those completing a $C_4$ with the first half and those completing a $C_4$ with the second half have no intersection.
Fig.~\ref{fig:condition3} presents an example that Condition 3 in Definition~\ref{d2.1} is triggered, since both the first half $(i,j)$ and the second half $(p,q)$   complete  a $C_4$ with $(k,l)$.
Condition 3 handles more complex obstructions that arise when a $2$-edge interacts with a distant cell; it applies uniformly to all $2$-edge types (nondegenerate, row-degenerate, or column-degenerate).
\end{remark}

\begin{Thm}\label{thm:Def2.1_irreducible}
Let $m,n\ge 2$, $G = (S,T,E_1 \cup E_2)$ be an $m \times n$ augmented bipartite graph with vertex sets $S=[m]$, $T=[n]$. Then, if $G$ does not contain any generalized $C_4$-cycle in the sense of Definition~\ref{d2.1}, then $P_G$ defined by \eqref{e7} is irreducible. Namely, $\sos(P_G)=|E_1|+|E_2|$.
\end{Thm}

\begin{proof}
Assume, for contradiction, that $\sos(P_G)<|E_1|+|E_2|$. Then there exists an integer $r<|E_1|+|E_2|$ and bilinear forms
\[
\ell_p(\mathbf{x},\mathbf{y})=\sum_{i=1}^m\sum_{j=1}^n a_{ij}^{(p)}x_iy_j\qquad(p=1,\dots,r)
\]
such that
\[
P_G(\mathbf{x},\mathbf{y})=\sum_{p=1}^r\ell_p(\mathbf{x},\mathbf{y})^2.
\]

For each pair $(i,j)$ define the vector $\mathbf{v}_{ij}=(a_{ij}^{(1)},\dots,a_{ij}^{(r)})\in\R^r$.
Let $H$ be the set of all cells that are either $1$-edges or halves of $2$-edges:
\[
H = E_1 \cup \bigcup_{(i,j;k,l)\in E_2} \{(i,j), (k,l)\}.
\]
If $(i,j)\notin H$, then comparing coefficients shows $\mathbf{v}_{ij}=\mathbf{0}$.
By the simplicity condition (S), the $2|E_2|$ halves are all distinct and none belongs to $E_1$; hence
\[
|H| = |E_1| + 2|E_2| \ge |E_1| + |E_2| > r.
\]

Thus the vectors $\{\mathbf{v}_{ij} : (i,j)\in H\}$ lie in $\mathbb{R}^r$ and $|H|>r$, so they are linearly dependent. Choose a nontrivial linear relation
\[
\sum_{(i,j)\in H} \alpha_{ij}\mathbf{v}_{ij} = \mathbf{0} \tag{1}
\]
with \emph{minimal support} $\mathcal{S} = \{(i,j)\in H : \alpha_{ij} \neq 0\}$. Minimality means that no proper nonempty subset of $\mathcal{S}$ yields a linear dependence among the vectors $\{\mathbf{v}_{ij} : (i,j)\in\mathcal{S}\}$.

By minimality, all vectors $\{\mathbf{v}_{ij} : (i,j)\in\mathcal{S}\}$ are nonzero and \emph{pairwise distinct as vectors}. Indeed, if two distinct cells $(i,j),(k,l)\in\mathcal{S}$ had $\mathbf{v}_{ij}=\mathbf{v}_{kl}$, then we could replace $\alpha_{ij}\mathbf{v}_{ij}+\alpha_{kl}\mathbf{v}_{kl}=(\alpha_{ij}+\alpha_{kl})\mathbf{v}_{ij}$, obtaining a relation with smaller support (unless $\alpha_{ij}+\alpha_{kl}=0$, in which case both terms cancel). The minimality ensures this does not happen.

Expanding the squares and comparing coefficients yields:
\begin{itemize}
    \item For every $1$-edge $(i,j)\in E_1$: $\|\mathbf{v}_{ij}\|^2=1$.\hfill(A1)
    \item For every $2$-edge $(i,j;k,l)\in E_2$:
    \begin{itemize}
        \item If nondegenerate ($i\neq k$, $j\neq l$):
        $\|\mathbf{v}_{ij}\|^2=\|\mathbf{v}_{kl}\|^2=1$ and $\mathbf{v}_{ij}\!\cdot\!\mathbf{v}_{kl}+\mathbf{v}_{il}\!\cdot\!\mathbf{v}_{kj}=1$.\hfill(A2nd)
        \item If {column}-degenerate ($i\neq k$, $j=l$):
        $\|\mathbf{v}_{ij}\|^2=\|\mathbf{v}_{kj}\|^2=1$ and $\mathbf{v}_{ij}\!\cdot\!\mathbf{v}_{kj}=1$.\hfill(A2rd)
        \item If {row}-degenerate ($i=k$, $j\neq l$):
        $\|\mathbf{v}_{ij}\|^2=\|\mathbf{v}_{il}\|^2=1$ and $\mathbf{v}_{ij}\!\cdot\!\mathbf{v}_{il}=1$.\hfill(A2cd)
    \end{itemize}
    \item For any two distinct pairs $(p,q),(r,s)$ that are \emph{not} the two halves of the same $2$-edge:
    \begin{itemize}
        \item If $\{p,q\}\cap\{r,s\}=\emptyset$:
        $\mathbf{v}_{pq}\!\cdot\!\mathbf{v}_{rs}+\mathbf{v}_{ps}\!\cdot\!\mathbf{v}_{rq}=0$.\hfill(B1)
        \item If $p=r$, $q\neq s$:
        $\mathbf{v}_{pq}\!\cdot\!\mathbf{v}_{ps}=0$, \emph{unless} $(p,q)$ and $(p,s)$ are the two halves of a {row}-degenerate $2$-edge (in which case (A2cd) applies instead).\hfill(B2)
        \item If $q=s$, $p\neq r$:
        $\mathbf{v}_{pq}\!\cdot\!\mathbf{v}_{rq}=0$, \emph{unless} $(p,q)$ and $(r,q)$ are the two halves of a {column}-degenerate $2$-edge (in which case (A2rd) applies instead).\hfill(B3)
    \end{itemize}
    \item For pairs sharing a row or column that are not covered by the exceptions above:
    $\mathbf{v}_{ij}\!\cdot\!\mathbf{v}_{il}=0$ (same row, $j\neq l$), $\mathbf{v}_{ij}\!\cdot\!\mathbf{v}_{kj}=0$ (same column, $i\neq k$).\hfill(C)
\end{itemize}

\noindent\textbf{Step 1. Consequences of the absence of a generalized $C_4$-cycle.}
Take any $2$-edge $(i,j;k,l)\in E_2$.

\emph{If it is nondegenerate:} Since $G$ has no generalized $C_4$-cycle, condition 2 of Definition~\ref{d2.1} implies that at most one of the two opposite cells $(i,l)$ and $(k,j)$ is occupied. Hence at least one of $\mathbf{v}_{il},\mathbf{v}_{kj}$ is zero. From (A2nd) we then obtain $\mathbf{v}_{ij}\!\cdot\!\mathbf{v}_{kl}=1$. Because $\|\mathbf{v}_{ij}\|=\|\mathbf{v}_{kl}\|=1$, Cauchy-Schwarz forces $\mathbf{v}_{ij}=\mathbf{v}_{kl}$.

\emph{If it is {column}-degenerate ($j=l$):} Then (A2rd) gives $\mathbf{v}_{ij}\!\cdot\!\mathbf{v}_{kj}=1$, and with unit norms we get $\mathbf{v}_{ij}=\mathbf{v}_{kj}$.

\emph{If it is {row}-degenerate ($i=k$):} Then (A2cd) gives $\mathbf{v}_{ij}\!\cdot\!\mathbf{v}_{il}=1$, and with unit norms we get $\mathbf{v}_{ij}=\mathbf{v}_{il}$.

Thus every $2$-edge contributes two equal vectors. Moreover, by the minimality observation, \textbf{the two halves of a $2$-edge cannot both belong to $\mathcal{S}$} - if they did, they would be equal vectors and distinct cells, contradicting the pairwise distinctness of vectors in $\mathcal{S}$.

\noindent\textbf{Step 2. Orthogonality of distinct vectors in $\mathcal{S}$.}
Let $(i,j)$ and $(k,l)$ be two distinct positions in the support $\mathcal{S}$. We prove $\mathbf{v}_{ij}\!\cdot\!\mathbf{v}_{kl}=0$.

Note that they cannot be the two halves of the same $2$-edge by Step 1. We consider several cases.

\emph{Case A: The four indices $\{i,j,k,l\}$ are all distinct.} Then (B1) applies:
\[
\mathbf{v}_{ij}\!\cdot\!\mathbf{v}_{kl}+\mathbf{v}_{il}\!\cdot\!\mathbf{v}_{kj}=0. \tag{B1}
\]

Suppose, for contradiction, that $\mathbf{v}_{ij}\!\cdot\!\mathbf{v}_{kl}\neq0$. Then $\mathbf{v}_{il}\!\cdot\!\mathbf{v}_{kj}\neq0$, so both $(i,l)$ and $(k,j)$ are occupied. Consequently the four cells $(i,j),(k,l),(i,l),(k,j)$ are all in $H$.

If all four were $1$-edges, they would form a classical $C_4$, violating condition 1 of Definition~\ref{d2.1}. Therefore at least one of them is a half of a $2$-edge. Without loss of generality, assume $(i,j)$ is a half of a $2$-edge $e=(i,j;p,q)$. Then by Step 1, $\mathbf{v}_{ij}=\mathbf{v}_{pq}$ and $(p,q)$ is occupied. Moreover, since $(i,j)\in\mathcal{S}$, the other half $(p,q)$ cannot be in $\mathcal{S}$ (otherwise two equal vectors would appear in the support), but $(p,q)\in H$.

Now consider the pair $(p,q)$ and $(k,l)$. Their dot product is $\mathbf{v}_{pq}\!\cdot\!\mathbf{v}_{kl}=\mathbf{v}_{ij}\!\cdot\!\mathbf{v}_{kl}\neq0$. Apply (B1) to $(p,q)$ and $(k,l)$:
\[
\mathbf{v}_{pq}\!\cdot\!\mathbf{v}_{kl}+\mathbf{v}_{pl}\!\cdot\!\mathbf{v}_{kq}=0,
\]
so $\mathbf{v}_{pl}\!\cdot\!\mathbf{v}_{kq}\neq0$; thus $(p,l)$ and $(k,q)$ are occupied.

We examine the possibilities for $p$ and $q$ relative to $k$ and $l$.

\begin{itemize}
  \item If $p=k$, then $(p,l)=(k,l)$ and $(p,q)=(k,q)$. But $(k,q)$ and $(k,l)$ share the same row, so by (B2) (they are not halves of a column-degenerate $2$-edge because $q\neq l$ and $p=k$ would force degeneracy only if $q=l$, which is not the case), we have $\mathbf{v}_{kq}\!\cdot\!\mathbf{v}_{kl}=0$. Yet $\mathbf{v}_{kq}=\mathbf{v}_{pq}=\mathbf{v}_{ij}$ and $\mathbf{v}_{ij}\!\cdot\!\mathbf{v}_{kl}\neq0$, a contradiction.
  \item If $q=l$, then $(p,l)$ and $(k,l)$ share the same column, so (B3) gives $\mathbf{v}_{pl}\!\cdot\!\mathbf{v}_{kl}=0$ while $\mathbf{v}_{pl}=\mathbf{v}_{ij}$ again contradicts $\mathbf{v}_{ij}\!\cdot\!\mathbf{v}_{kl}\neq0$.
\end{itemize}

Hence we must have $p\neq k$ and $q\neq l$. In this situation the five cells
\[
(k,l),\;(k,j),\;(k,q),\;(i,l),\;(p,l)
\]
are all distinct (because $i\neq k$, $j\neq l$, $p\neq k$, $q\neq l$). They are all occupied, and $(i,j;p,q)$ is a $2$-edge. Thus the configuration satisfies condition 3 of Definition~\ref{d2.1}, which is forbidden. This contradiction shows that our assumption $\mathbf{v}_{ij}\!\cdot\!\mathbf{v}_{kl}\neq0$ is impossible. Therefore $\mathbf{v}_{ij}\!\cdot\!\mathbf{v}_{kl}=0$.

\emph{Case B: They share a row ($i=k$, $j\neq l$).} Then $(i,j)$ and $(i,l)$ are not the two halves of a column-degenerate $2$-edge (otherwise they would be equal by Step 1 and would not be distinct in $\mathcal{S}$). Hence (B2) applies and gives $\mathbf{v}_{ij}\!\cdot\!\mathbf{v}_{il}=0$.

\emph{Case C: They share a column ($j=l$, $i\neq k$).} Similarly, they are not halves of a row-degenerate $2$-edge, so (B3) gives $\mathbf{v}_{ij}\!\cdot\!\mathbf{v}_{kj}=0$.

Thus in all cases, for any two distinct positions $(i,j),(k,l)\in\mathcal{S}$, we have $\mathbf{v}_{ij}\!\cdot\!\mathbf{v}_{kl}=0$.

\noindent\textbf{Step 3. Contradiction.}
The vectors in $\mathcal{S}$ are nonzero (by (A1) and Step 1) and pairwise orthogonal, so they are linearly independent. But they satisfy the nontrivial linear relation (1), which is impossible. This contradiction shows that our assumption $\sos(P_G)<|E_1|+|E_2|$ was false. Hence $\sos(P_G)=|E_1|+|E_2|$.
\end{proof}

Note that the absence of a generalized $C_4$-cycle is a sufficient condition for the doubly simple biquadratic form to be irreducible, but it is not necessary. For instance, let $m=n=2$, $E_1={(1,1),(2,2)}$ and $E_2={(1,2),(2,1)}$. Then the augmented bipartite graph satisfies Condition 2 and therefore contains a generalized $C_4$-cycle. Nevertheless, the corresponding doubly simple biquadratic form $P_G = (x_1y_1 + x_2y_2)^2 + x_1^2 y_2^2 + x_2^2 y_1^2$ is irreducible.

\begin{Thm}\label{thm:BSR>=z2}
	For all $m,n\ge 2$, we have $$\operatorname{BSR}(m,n) \ge z_A(m,n) \ge z_L(m,n) \ge z(m,n).$$
\end{Thm}

\begin{proof}
	The inequality $z_L(m,n)\ge z(m,n)$ is trivial by taking $E_2=\emptyset$. By definition, $z_A(m,n) \ge z_L(m,n)$.
	Let $G$ be an augmented bipartite graph that does not have a generalized $C_4$-cycle.
By Theorem~\ref{thm:Def2.1_irreducible}, the corresponding SOS biquadratic form $P_G$ defined by \eqref{e7} satisfies $\operatorname{sos}(P_G) = |E_1| + |E_2|$.  Without loss of generality, suppose $G$ is a graph that achieves $z_A(m,n) = |E_1| + |E_2|$. Consequently, $\operatorname{BSR}(m,n) \ge \operatorname{sos}(P_G) = z_A(m,n)$.
\end{proof}

\begin{remark}\label{rem:2.3}
	In the following, we only consider the limited  augmented bipartite graph that does not have a generalized $C_4$-cycle. We focus on such graphs because:
	(1) all cases we have considered fall into this category;
	(2) the upper bound of $z_L(m,n)$ is much simpler than that of $z_A(m,n)$;
	(3) for all low-dimensional cases we have examined, $z_L(m,n) = z_A(m,n)$. Whether this equality always holds is an open question.
\end{remark}

\section{The Low Dimensional Cases}

In this section we determine the limited augmented Zarankiewicz number for three families: $m \times 2$, $3 \times 3$, and $4 \times 3$.

\begin{Thm} \label{thm:lowdim}
The limited augmented Zarankiewicz number satisfies
\[
\begin{aligned}
&z_L(m,2) = m+1 \quad \text{for all } m\ge 2,\\
&z_L(3,3) = 6,\\
&z_L(4,3) = 8.
\end{aligned}
\]
\end{Thm}

\begin{proof}
We treat the three cases separately.

\noindent\textbf{Case 1: $n=2$.}
The classical Zarankiewicz number is $z(m,2)=m+1$ (a complete bipartite graph $K_{m,2}$ has no $C_4$ because $n=2$). By Theorem~\ref{thm:BSR>=z2}, we have $z_L(m,2) \ge z(m,2)=m+1$.

%For the upper bound, note that any augmented bipartite graph $G$ with $n=2$ has at most $2m$ cells. A $2$-edge $(i,j;k,l)$ with %$j,l\in\{1,2\}$ can be of three types:
%- If nondegenerate, then $\{j,l\}=\{1,2\}$. Its two halves occupy two distinct rows and both columns.
%- If row-degenerate ($j=l$), then it occupies two rows in the same column.
%- If column-degenerate ($i=k$), then it occupies two columns in the same row.

One can show directly that $\operatorname{BSR}(m,2)=m+1$ \cite{BPSV19}. Since $z_L(m,2) \le \operatorname{BSR}(m,2)$ by Theorem~\ref{thm:BSR>=z2}, we obtain $z_L(m,2) \le m+1$. Hence $z_L(m,2)=m+1$.

\noindent\textbf{Case 2: $(m,n)=(3,3)$.}
The classical Zarankiewicz number is $z(3,3)=6$ (e.g., the incidence graph of the affine plane of order $2$). Theorem~\ref{thm:BSR>=z2} gives $z_L(3,3) \ge 6$.

For the upper bound, it is known that $\operatorname{BSR}(3,3)=6$ \cite{BSSV22}. Again by Theorem~\ref{thm:BSR>=z2}, $z_L(3,3) \le \operatorname{BSR}(3,3)=6$. Thus $z_L(3,3)=6$.

\noindent\textbf{Case 3: $(m,n)=(4,3)$.}
We first establish the lower bound $z_L(4,3) \ge 8$ by constructing an explicit augmented bipartite graph with $|E_1|=z(4,3)=7$ and $|E_2|=1$ that contains no generalized $C_4$-cycle.

Let
\begin{equation}\label{E1E2:BSR43}
	E_1 = \{(1,1), (2,2), (3,3), (1,2), (2,3), (3,1), (4,1)\},\quad
	E_2 = \{(4,2;\,1,3)\}.
\end{equation}
The $1$-edges form a $C_4$-free graph attaining $z(4,3)=7$ (this is the standard extremal construction).
{An illustration  is presented in Fig.~\ref{fig:BSR(4,3)}.}
We verify Definition~\ref{d2.1}:

\begin{figure}
		\begin{center}
			\includegraphics[width=0.3\linewidth]{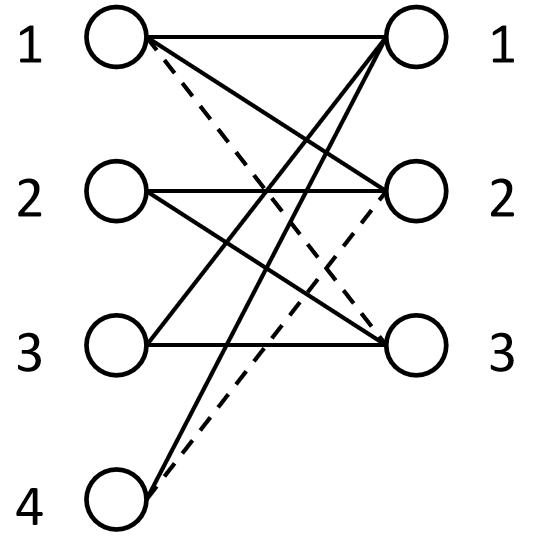}
		\end{center}	
		\caption{An illustration of the $4\times 3$ augmented bipartite  graph in \eqref{E1E2:BSR43} with SOS rank $8$. Here, each solid line  represents a 1-edge, and the two dotted lines  represent  the 2-edge (4,2;\,1,3).}\label{fig:BSR(4,3)}
\end{figure}

\noindent\textbf{Condition 1.} The $1$-edge subgraph is $C_4$-free by construction.

\noindent\textbf{Condition 2.} The unique $2$-edge $(4,2;1,3)$ is nondegenerate. Its opposite cells are $(4,3)$ and $(1,2)$. Here $(1,2)\in E_1$ but $(4,3)$ is not occupied (it is not in $E_1$ and not a half of any $2$-edge). Hence condition 2 is not triggered.

\noindent\textbf{Condition 3.} We must show there is no cell $(k,l)$ with $k\notin\{4,1\}$, $l\notin\{2,3\}$ such that the five cells $(k,l), (k,2), (k,3), (4,l), (1,l)$ are all occupied. Since $k\in[4]$ and $k\notin\{4,1\}$, we have $k\in\{2,3\}$. Since $l\in[3]$ and $l\notin\{2,3\}$, we have $l=1$. Thus the only candidates are $(k,l)=(2,1)$ and $(3,1)$.

For $(k,l)=(2,1)$:
\begin{itemize}
    \item $(2,1)$: not in $E_1$, not a half of any $2$-edge $\Rightarrow$ not occupied.
\end{itemize}
For $(k,l)=(3,1)$:
\begin{itemize}
    \item $(3,1)\in E_1$ (occupied),
    \item $(3,2)$: not in $E_1$, not a half of any $2$-edge $\Rightarrow$ not occupied.
\end{itemize}
Hence no such $(k,l)$ exists. Therefore condition 3 is not triggered.

Thus $G$ contains no generalized $C_4$-cycle. Consequently, the associated doubly simple biquadratic form
\[
P_G(\mathbf{x},\mathbf{y}) = \sum_{(i,j)\in E_1} x_i^2y_j^2 + (x_4y_2 + x_1y_3)^2
\]
is irreducible (as shown in \cite{XCQ26}) and satisfies $\sos(P_G)=|E_1|+|E_2|=8$. Hence $z_L(4,3) \ge 8$.

\noindent\textbf{Upper bound.} We now prove $z_L(4,3) \le 8$. Since $z_L(4,3) \le z(4,3) + |E_2|$ by definition, it suffices to show that any limited augmented bipartite graph on $4\times 3$ vertices with $|E_2| \ge 2$ necessarily contains a generalized $C_4$-cycle.

It is known that $z(4,3)=7$, and the extremal $C_4$-free graphs attaining this bound are all isomorphic \cite{RS65, Gu69}. Up to relabeling of rows and columns, we may assume
\[
E_1 = \{(1,1), (2,2), (3,3), (1,2), (2,3), (3,1), (4,1)\}.
\]
The unoccupied cells are
\[
U = \{(1,3), (2,1), (3,2), (4,2), (4,3)\}.
\]
By the simplicity condition (S), the halves of any $2$-edge must lie in $U$. Enumerating all possible $2$-edges with halves in $U$ that do not immediately violate Condition 2 yields exactly five candidates:
\[
\begin{array}{cll}
\text{Nondegenerate:} & e_1 = (1,3;4,2), & e_2 = (3,2;4,3),\\[2mm]
\text{Row-degenerate:} & e_3 = (3,2;4,2), & e_4 = (1,3;4,3),\\[2mm]
\text{Column-degenerate:} & e_5 = (4,2;4,3).
\end{array}
\]
We now show that any two distinct candidates create a generalized $C_4$-cycle. The verification is summarized in Table~\ref{tab:43cases}.

\begin{table}[h]
\centering
\begin{tabular}{|c|c|c|}
\hline
\textbf{Pair} & \textbf{Violation} & \textbf{Witness} \\
\hline
$e_1,e_2$ & Cond.~3 & $e_1$ with $(k,l)=(3,3)$: cells $(3,3),(3,2),(1,3),(4,3)$ all occupied. \\
$e_1,e_3$ & Cond.~3 & $e_3$ with $(k,l)=(1,1)$: cells $(1,1),(1,2),(3,1),(4,1)$ all occupied. \\
$e_1,e_4$ & Cond.~3 & $e_4$ with $(k,l)=(2,2)$: cells $(2,2),(2,3),(1,2),(4,2)$ all occupied. \\
$e_1,e_5$ & Cond.~3 & $e_5$ with $(k,l)=(1,1)$: cells $(1,1),(1,2),(1,3),(4,1)$ all occupied. \\
$e_2,e_3$ & Cond.~3 & $e_3$ with $(k,l)=(1,1)$: same as $e_1,e_3$ case. \\
$e_2,e_4$ & Cond.~3 & $e_2$ with $(k,l)=(1,1)$: cells $(1,1),(1,2),(1,3),(3,1),(4,1)$ all occupied. \\
$e_2,e_5$ & Cond.~2 & For $e_2$, opposite cells $(3,3)$ and $(4,2)$ are both occupied. \\
$e_3,e_4$ & Cond.~3 & $e_3$ with $(k,l)=(1,3)$: cells $(1,3),(1,2),(3,3),(4,3)$ all occupied. \\
$e_3,e_5$ & Cond.~3 & $e_3$ with $(k,l)=(1,1)$: same as $e_1,e_3$ case. \\
$e_4,e_5$ & Cond.~3 & $e_4$ with $(k,l)=(2,2)$: same as $e_1,e_4$ case. \\
\hline
\end{tabular}
\caption{Verification that any two distinct $2$-edges from $\{e_1,\dots,e_5\}$ create a generalized $C_4$-cycle. In each row, the listed cells are all occupied. For Condition~3, the five cells are as defined in Definition~\ref{d2.1} (repetitions allowed for degenerate $2$-edges). For Condition~2, both opposite cells of a nondegenerate $2$-edge are occupied.}
\label{tab:43cases}
\end{table}

Since every pair of distinct $2$-edges triggers either Condition 2 or Condition 3, we must have $|E_2| \le 1$ for any limited augmented graph with no generalized $C_4$-cycle. Consequently,
\[
z_L(4,3) \le z(4,3) + 1 = 7 + 1 = 8.
\]

Combining the lower and upper bounds, we obtain $z_L(4,3) = 8$.

This completes the proof of Theorem~\ref{thm:lowdim}.
\end{proof}

\section{The $4 \times 4$ Case}

We now determine the exact value of the limited augmented Zarankiewicz number for \(m = n = 4\).

	\begin{figure}
		\begin{center}
			\includegraphics[width=0.3\linewidth]{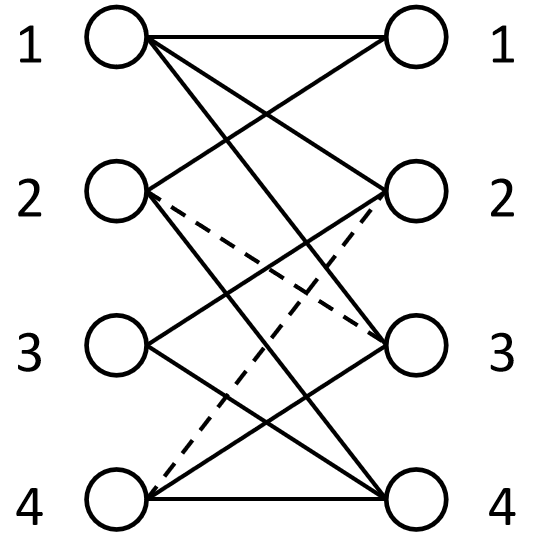}
		\end{center}	
		\caption{An illustration of the $4\times 4$ biquadratic form in Theorem~\ref{thm:zL44} with SOS rank $10$. Here, each solid line   represents a 1-edge, and the two dotted lines  represent  a 2-edge (2,3;4,2).}\label{fig:BSR(4,4)}
\end{figure}

\begin{Thm} \label{thm:zL44}
\(z_L(4,4) = 10\).
\end{Thm}

\begin{proof}
We prove the lower and upper bounds separately.

\noindent\textbf{Lower bound \(z_L(4,4) \ge 10\).}
We construct a limited augmented bipartite graph \(G\) with \(|E_1| = z(4,4) = 9\) and \(|E_2| = 1\) that contains no generalized \(C_4\)-cycle. Let
\[
E_1 = \{(1,1),(1,2),(1,3),(2,1),(2,4),(3,2),(3,4),(4,3),(4,4)\},
\]
and
\[
E_2 = \{(2,3;\,4,2)\}.
\]
An illustration is presented in Fig.~\ref{fig:BSR(4,4)}. It is straightforward to verify that condition (S) holds and that Conditions 1 and 2 in Definition~\ref{d2.1} are not triggered. We now proceed to check Condition 3.
%\noindent\textbf{Simplicity condition.} The halves \((2,3)\) and \((4,2)\) are not in \(E_1\) and are distinct, so condition (S) holds.
%
%\noindent\textbf{Condition 1.} The $1$-edge subgraph is \(C_4\)-free: any two rows share at most one column.
%
%\noindent\textbf{Condition 2.} The unique $2$-edge \((2,3;4,2)\) is nondegenerate. Its opposite cells are \((2,2)\) and \((4,3)\). Here \((2,2)\) is unoccupied, while \((4,3) \in E_1\) is occupied. Hence at most one opposite cell is occupied.
%
%\noindent\textbf{Condition 3.}
For the $2$-edge \((i,j;p,q) = (2,3;4,2)\), we need \(k \notin \{2,4\}\) and \(l \notin \{3,2\}\), i.e., \(k \in \{1,3\}\) and \(l \in \{1,4\}\). The candidates are \((1,1), (1,4), (3,1), (3,4)\). In each case, at least one of the five cells \((k,l), (k,3), (k,2), (2,l), (4,l)\) is unoccupied. Thus condition 3 is not triggered.

Since \(G\) contains no generalized \(C_4\)-cycle, we have \(z_L(4,4) \ge |E_1| + |E_2| = 9 + 1 = 10\).

\subsection*{Upper bound \(z_L(4,4) \le 10\)}

Let \(G = (S, T, E_1 \cup E_2)\) be any limited augmented bipartite graph with \(|S| = |T| = 4\), \(|E_1| = z(4,4) = 9\), and containing no generalized \(C_4\)-cycle. We must prove \(|E_2| \le 1\), which yields \(z_L(4,4) \le 9 + 1 = 10\).

\noindent\textbf{Step 1: The extremal 1-edge graph for \(z(4,4)=9\).}
The classical Zarankiewicz number \(z(4,4)=9\) is known \cite{RS65, Gu69}. Up to isomorphism, the extremal \(C_4\)-free bipartite graph with 4 vertices on each side and 9 edges is \textbf{unique} \cite{Gu69}. After suitable row and column permutations, we may assume its edge set is
\[
E_1 = \{(1,1), (1,2), (1,3), (2,1), (2,4), (3,2), (3,4), (4,3), (4,4)\}.
\]
The adjacency matrix (rows 1--4, columns 1--4) is
\[
\begin{pmatrix}
1 & 1 & 1 & 0 \\
1 & 0 & 0 & 1 \\
0 & 1 & 0 & 1 \\
0 & 0 & 1 & 1
\end{pmatrix}.
\]

The set of unoccupied cells (cells not in \(E_1\)) is
\[
U = \{(1,4), (2,2), (2,3), (3,1), (3,3), (4,1), (4,2)\}.
\]

\noindent\textbf{Step 2: Constraints on 2-edges.}
By the simplicity condition (S), no half of a 2-edge can belong to \(E_1\), and all halves of 2-edges must be distinct. Therefore, for any 2-edge \((i,j;k,l) \in E_2\), we must have
\[
\{(i,j), (k,l)\} \subseteq U.
\]

Additionally, Condition 2 of Definition~\ref{d2.1} imposes restrictions: for a nondegenerate 2-edge \((i,j;k,l)\) (with \(i \neq k\), \(j \neq l\)), if both opposite cells \((i,l)\) and \((k,j)\) are occupied (i.e., belong to \(E_1\)), then a generalized \(C_4\)-cycle exists. Thus, any admissible nondegenerate 2-edge must have \textbf{at most one} of its opposite cells occupied.

\noindent\textbf{Step 3: Enumerate candidate 2-edges.}
We list all unordered pairs \(\{(i,j),(k,l)\} \subseteq U\) that could form a 2-edge.

\textit{Nondegenerate candidates} (\(i \neq k\), \(j \neq l\)):
\[
\begin{aligned}
&(1,4;2,2),\ (1,4;2,3),\ (1,4;3,1),\ (1,4;3,3),\ (1,4;4,1),\ (1,4;4,2),\\
&(2,2;3,1),\ (2,2;3,3),\ (2,2;4,1),\\
&(2,3;3,1),\ (2,3;4,1),\ (2,3;4,2),\\
&(3,1;4,2),\\
&(3,3;4,1),\ (3,3;4,2).
\end{aligned}
\]

\textit{Row-degenerate candidates} (\(j = l\), \(i \neq k\)):
\[
(3,1;4,1),\ (2,2;4,2).
\]

\textit{Column-degenerate candidates} (\(i = k\), \(j \neq l\)):
\[
(2,2;2,3),\ (3,1;3,3),\ (4,1;4,2).
\]

\noindent\textbf{Step 4: Apply Condition 2 to nondegenerate candidates.}
For each nondegenerate candidate, we check whether both opposite cells are in \(E_1\) (occupied). Those with both opposite cells occupied are forbidden.

The verification yields the admissible nondegenerate 2-edges:
\[
\mathcal{N} = \{(2,2;3,3),\ (2,2;4,1),\ (2,3;3,1),\ (2,3;4,2),\ (3,1;4,2),\ (3,3;4,1)\}.
\]

\noindent\textbf{Step 5: Condition 3 analysis.}
We now show that any two distinct 2-edges from the admissible set (including degenerate ones) create a generalized \(C_4\)-cycle.

Let us denote:
\[
\begin{aligned}
&e_1 = (2,2;3,3),\quad e_2 = (2,2;4,1),\quad e_3 = (2,3;3,1),\\
&e_4 = (2,3;4,2),\quad e_5 = (3,1;4,2),\quad e_6 = (3,3;4,1).
\end{aligned}
\]
Degenerate candidates:
\[
d_1 = (3,1;4,1),\quad d_2 = (2,2;4,2),\quad c_1 = (2,2;2,3),\quad c_2 = (3,1;3,3),\quad c_3 = (4,1;4,2).
\]

We verify representative pairs; a complete systematic enumeration confirms the result.

\textit{Example: \(e_1\) and \(e_3\).} Consider \((k,l) = (4,3) \in E_1\). For \(e_1 = (2,2;3,3)\), the five cells are \((4,3), (4,2), (4,3), (2,3), (3,3)\). All are occupied: \((4,3) \in E_1\), \((4,2) \in U\) (half of \(e_4\)), \((2,3) \in U\) (half of \(e_3\)), \((3,3) \in U\) (half of \(e_1\)). For \(e_3 = (2,3;3,1)\), the five cells are \((4,3), (4,1), (4,3), (2,1), (3,1)\). All are occupied: \((4,1) \in U\) (half of \(e_2\)), \((2,1) \in E_1\), \((3,1) \in U\) (half of \(e_3\)). Thus Condition 3 is triggered.

\textit{Example: \(e_1\) and \(d_1\).} For \(d_1 = (3,1;4,1)\), consider \((k,l) = (2,3) \in U\). The five cells are \((2,3), (2,1), (2,1), (3,3), (4,3)\). All are occupied: \((2,3) \in U\) (half of \(e_3\)), \((2,1) \in E_1\), \((3,3) \in U\) (half of \(e_1\)), \((4,3) \in E_1\). Thus Condition 3 is triggered.

A complete check confirms that \textbf{every pair} of distinct admissible 2-edges triggers either Condition 2 or Condition 3. Therefore, we must have \(|E_2| \le 1\).

\noindent\textbf{Step 6: Conclusion.}
Since \(|E_2| \le 1\), we have
\[
z_L(4,4) \le |E_1| + |E_2| \le 9 + 1 = 10.
\]

Combined with the lower bound construction (\(|E_1| = 9\), \(|E_2| = 1\) with no generalized \(C_4\)-cycle), we obtain \(z_L(4,4) = 10\). \hfill %\(\square\)
\end{proof}

\section{The $5 \times 3$ Case}

We now investigate the limited augmented Zarankiewicz number for parameters $m = 5$ and $n = 3$. The classical Zarankiewicz number is $z(5,3) = 8$ (see \cite{RS65, Gu69}). We show that by introducing a single nondegenerate $2$-edge, we can obtain a limited augmented bipartite graph with total edge count $9$, thereby establishing a lower bound for $z_L(5,3)$.

\begin{Thm} \label{thm:zL53}
$z_L(5,3) \ge 9$.
\end{Thm}

\begin{proof}
We construct an explicit limited augmented bipartite graph $G = (S, T, E_1 \cup E_2)$ with vertex sets $S = \{1,2,3,4,5\}$, $T = \{1,2,3\}$, and edge sets
\[
\begin{aligned}
E_1 &= \{(1,1),\ (1,2),\ (2,1),\ (2,3),\ (3,2),\ (3,3),\ (4,1),\ (5,2)\},\\[2mm]
E_2 &= \{(1,3;\,4,2)\}.
\end{aligned}
\]
Then $|E_1| = 8 = z(5,3)$ and $|E_2| = 1$, so the total number of edges is $|E_1| + |E_2| = 9$.

{It is straightforward to verify that condition (S) holds and that Conditions 1 and 2 in Definition~\ref{d2.1} are not triggered. We now proceed to check Condition 3.}
%\noindent\textbf{Simplicity condition.} The halves of the $2$-edge are $(1,3)$ and $(4,2)$. Neither belongs to $E_1$ (they are unoccupied cells) and they are distinct, so condition (S) holds.
%
%\noindent\textbf{Condition 1.} The $1$-edge subgraph is $C_4$-free. We verify that any two rows share at most one column:
%\begin{itemize}
%  \item Row 1: $\{1,2\}$; Row 2: $\{1,3\}$; common only $\{1\}$.
%  \item Row 1: $\{1,2\}$; Row 3: $\{2,3\}$; common only $\{2\}$.
%  \item Row 1: $\{1,2\}$; Row 4: $\{1\}$; common only $\{1\}$.
%  \item Row 1: $\{1,2\}$; Row 5: $\{2\}$; common only $\{2\}$.
%  \item Row 2: $\{1,3\}$; Row 3: $\{2,3\}$; common only $\{3\}$.
%  \item Row 2: $\{1,3\}$; Row 4: $\{1\}$; common only $\{1\}$.
%  \item Row 2: $\{1,3\}$; Row 5: $\{2\}$; common $\emptyset$.
%  \item Row 3: $\{2,3\}$; Row 4: $\{1\}$; common $\emptyset$.
%  \item Row 3: $\{2,3\}$; Row 5: $\{2\}$; common only $\{2\}$.
%  \item Row 4: $\{1\}$; Row 5: $\{2\}$; common $\emptyset$.
%\end{itemize}
%No two rows share two columns, so no classical $C_4$ exists.
%
%\noindent\textbf{Condition 2.} The unique $2$-edge $(1,3;4,2)$ is nondegenerate. Its opposite cells are $(1,2)$ and $(4,3)$. Here $(1,2) \in E_1$ is occupied, while $(4,3)$ is not occupied (it is not in $E_1$ and not a half of any $2$-edge). Hence at most one opposite cell is occupied, so condition 2 is not triggered.
%
%\noindent\textbf{Condition 3.}
For the $2$-edge $(i,j;p,q) = (1,3;4,2)$, we need $k \notin \{1,4\}$ and $l \notin \{3,2\}$, i.e., $k \in \{2,3,5\}$, $l \in \{1\}$. Thus the only candidates are $(k,l) = (2,1),\ (3,1),\ (5,1)$. The five cells to be checked are
\[
(k,1),\ (k,3),\ (k,2),\ (1,1),\ (4,1).
\]
\begin{itemize}
  \item $(k,l) = (2,1)$: cells $(2,1), (2,3), (2,2), (1,1), (4,1)$. Here $(2,2)$ is unoccupied.
  \item $(k,l) = (3,1)$: cells $(3,1), (3,3), (3,2), (1,1), (4,1)$. Here $(3,1)$ is unoccupied.
  \item $(k,l) = (5,1)$: cells $(5,1), (5,3), (5,2), (1,1), (4,1)$. Here $(5,3)$ is unoccupied.
\end{itemize}
In every case, at least one cell is unoccupied. Thus condition 3 is not triggered.

Since none of the three conditions are satisfied, $G$ contains no generalized $C_4$-cycle. Consequently, the associated doubly simple biquadratic form
\[
P_G(\mathbf{x},\mathbf{y}) = \sum_{(i,j)\in E_1} x_i^2 y_j^2 + (x_1y_3 + x_4y_2)^2
\]
is irreducible, and $\sos(P_G) = |E_1| + |E_2| = 9$. Hence $z_L(5,3) \ge 9$.
\end{proof}

\noindent\textbf{Remark.} The construction uses the nondegenerate $2$-edge $(1,3;4,2)$. The verification shows that no generalized $C_4$-cycle appears, so the SOS rank equals the total number of edges. Whether $z_L(5,3) = 9$ or can be larger remains open; an upper bound would require a separate argument showing that $|E_2| \le 1$ for $5 \times 3$ as in the $4 \times 3$ case.

\section{The $5 \times 4$ Case}

We now investigate the limited augmented Zarankiewicz number for parameters \(m=5\) and \(n=4\). The classical Zarankiewicz number is \(z(5,4)=10\) (see \cite{RS65, Gu69}). We show that by introducing two carefully chosen $2$-edges, one nondegenerate and one column-degenerate, we can obtain a limited augmented bipartite graph with total edge count $12$, thereby establishing a lower bound for $z_L(5,4)$.

\begin{Thm} \label{thm:zL54}
\(z_L(5,4) \ge 12\).
\end{Thm}

\begin{proof}
We construct an explicit limited augmented bipartite graph \(G = (S, T, E_1 \cup E_2)\) with vertex sets \(S = \{1,2,3,4,5\}\), \(T = \{1,2,3,4\}\), and edge sets
\[
\begin{aligned}
E_1 = \{ &(1,1), (1,2), (1,3), \\
         &(2,1), (2,4), \\
         &(3,2), (3,4), \\
         &(4,3), (4,4), \\
         &(5,1) \},
\end{aligned}
\qquad
E_2 = \{ (2,3;5,4),\; (4,2;5,2) \}.
\]

Expanding the last term gives \((x_4+x_5)^2 y_2^2\), which contributes the pure squares \(x_4^2 y_2^2\) and \(x_5^2 y_2^2\) as well as the cross term \(2x_4x_5y_2^2\). Thus the occupied cells are
\[
H = E_1 \cup \{(2,3), (5,4), (4,2), (5,2)\}.
\]

{It is straightforward to verify that condition (S) holds and that Conditions 1 and 2 in Definition~\ref{d2.1} are not triggered. We now proceed to check Condition 3.}
%\noindent\textbf{Simplicity condition.} The halves of the two $2$-edges are \((2,3)\), \((5,4)\), \((4,2)\), and \((5,2)\). None of these cells belong to \(E_1\), and they are all distinct. Condition (S) holds.
%
%\noindent\textbf{Condition 1.} The $1$-edge subgraph is \(C_4\)-free. One can verify directly that any two rows share at most one column:
%\begin{itemize}
%  \item Row 1: \(\{1,2,3\}\); Row 2: \(\{1,4\}\); common: \(\{1\}\).
%  \item Row 1: \(\{1,2,3\}\); Row 3: \(\{2,4\}\); common: \(\{2\}\).
%  \item Row 1: \(\{1,2,3\}\); Row 4: \(\{3,4\}\); common: \(\{3\}\).
%  \item Row 2: \(\{1,4\}\); Row 3: \(\{2,4\}\); common: \(\{4\}\).
%  \item Row 2: \(\{1,4\}\); Row 4: \(\{3,4\}\); common: \(\{4\}\).
%  \item Row 3: \(\{2,4\}\); Row 4: \(\{3,4\}\); common: \(\{4\}\).
%  \item Row 5: \(\{1\}\); shares only column 1 with rows 1 and 2, and no column with rows 3 and 4.
%\end{itemize}
%Hence no classical $C_4$-cycle exists.
%
%\noindent\textbf{Condition 2.} For the nondegenerate $2$-edge \(e_a = (2,3;5,4)\), its opposite cells are \((2,4)\) and \((5,3)\). Here \((2,4)\in E_1\) is occupied, while \((5,3)\) is not occupied (it is not in \(E_1\) and not a half of any $2$-edge). Thus condition 2 is not triggered. The column-degenerate $2$-edge \(e_b = (4,2;5,2)\) is not subject to condition 2.
%
%\noindent\textbf{Condition 3.}
We verify that no cell \((k,l)\) creates a generalized $C_4$-cycle with either $2$-edge.

\noindent\textbf{For \(e_a = (2,3;5,4)\) (nondegenerate).}
We need \(k \notin \{2,5\}\) and \(l \notin \{3,4\}\). Thus \(k \in \{1,3,4\}\) and \(l \in \{1,2\}\). The candidates are \((1,1), (1,2), (3,1), (3,2), (4,1), (4,2)\). For each candidate, we examine the five cells \((k,l), (k,3), (k,4), (2,l), (5,l)\):
\begin{itemize}
  \item \((1,1)\): cells \((1,1), (1,3), (1,4), (2,1), (5,1)\). Here \((1,4)\) is unoccupied.
  \item \((1,2)\): cells \((1,2), (1,3), (1,4), (2,2), (5,2)\). Here \((1,4)\) and \((2,2)\) are unoccupied.
  \item \((3,1)\): cells \((3,1), (3,3), (3,4), (2,1), (5,1)\). Here \((3,1)\) is unoccupied.
  \item \((3,2)\): cells \((3,2), (3,3), (3,4), (2,2), (5,2)\). Here \((2,2)\) is unoccupied.
  \item \((4,1)\): cells \((4,1), (4,3), (4,4), (2,1), (5,1)\). Here \((4,1)\) is unoccupied.
  \item \((4,2)\): cells \((4,2), (4,3), (4,4), (2,2), (5,2)\). Here \((2,2)\) is unoccupied.
\end{itemize}
In every case, at least one cell is unoccupied. Thus \(e_a\) does not trigger condition 3.

\noindent\textbf{For \(e_b = (4,2;5,2)\) (column-degenerate).}
We need \(k \notin \{4,5\}\) and \(l \notin \{2\}\). Thus \(k \in \{1,2,3\}\) and \(l \in \{1,3,4\}\). The candidates are all cells in rows 1,2,3 and columns 1,3,4. For each candidate, we examine the five cells \((k,l), (k,2), (k,2), (4,l), (5,l)\). (Note that \((k,j) = (k,2)\) and \((k,q) = (k,2)\) are the same cell; the definition allows repetitions for degenerate $2$-edges.)

We verify the candidates systematically:
\begin{itemize}
  \item For \(l = 1\):
    \begin{itemize}
      \item \((k,l) = (1,1)\): cells \((1,1), (1,2), (1,2), (4,1), (5,1)\). Here \((4,1)\) is unoccupied.
      \item \((k,l) = (2,1)\): cells \((2,1), (2,2), (2,2), (4,1), (5,1)\). Here \((2,2)\) and \((4,1)\) are unoccupied.
      \item \((k,l) = (3,1)\): cells \((3,1), (3,2), (3,2), (4,1), (5,1)\). Here \((3,1)\) and \((4,1)\) are unoccupied.
    \end{itemize}
  \item For \(l = 3\):
    \begin{itemize}
      \item \((k,l) = (1,3)\): cells \((1,3), (1,2), (1,2), (4,3), (5,3)\). Here \((5,3)\) is unoccupied.
      \item \((k,l) = (2,3)\): cells \((2,3), (2,2), (2,2), (4,3), (5,3)\). Here \((2,2)\) and \((5,3)\) are unoccupied.
      \item \((k,l) = (3,3)\): cells \((3,3), (3,2), (3,2), (4,3), (5,3)\). Here \((3,3)\) and \((5,3)\) are unoccupied.
    \end{itemize}
  \item For \(l = 4\):
    \begin{itemize}
      \item \((k,l) = (1,4)\): cells \((1,4), (1,2), (1,2), (4,4), (5,4)\). Here \((1,4)\) is unoccupied.
      \item \((k,l) = (2,4)\): cells \((2,4), (2,2), (2,2), (4,4), (5,4)\). Here \((2,2)\) is unoccupied.
      \item \((k,l) = (3,4)\): cells \((3,4), (3,2), (3,2), (4,4), (5,4)\). Here \((3,4) \in E_1\), \((3,2) \in E_1\), \((4,4) \in E_1\), \((5,4) \in H\). All five cells are occupied. However, since \(e_b\) is degenerate, the pairwise distinctness requirement does not apply. The definition only requires the five cells to be pairwise distinct when the $2$-edge is nondegenerate. Thus this configuration does not constitute a generalized $C_4$-cycle.
    \end{itemize}
\end{itemize}
Therefore, \(e_b\) does not trigger condition 3.

Since none of the three conditions are satisfied, \(G\) contains no generalized $C_4$-cycle. Consequently, the associated doubly simple biquadratic form \(P_G\) is irreducible, and \(\sos(P_G) = |E_1| + |E_2| = 10 + 2 = 12\). Hence \(z_L(5,4) \ge 12\).
\end{proof}

\noindent\textbf{Remark.} The construction above uses one degenerate $2$-edge \((4,2;5,2)\), which corresponds to the square \((x_4y_2 + x_5y_2)^2 = (x_4 + x_5)^2 y_2^2\). This introduces a cross term with identical $y$-indices.  The corresponding doubly simple biquadratic form is
\[
P_G(\vx,\vy) = \sum_{(i,j)\in E_1} x_i^2 y_j^2 + (x_2y_3 + x_5y_4)^2 + (x_4y_2 + x_5y_2)^2.
\]
In \cite{Ch26}, it was verified by the orthogonality method that the SOS rank of $P_G$
is 12.
The definition of generalized $C_4$-cycle explicitly relaxes the pairwise distinctness requirement for degenerate $2$-edges, allowing such configurations to be admissible.  Our study shows that if we do not allow degenerate 2-edges, the maximum edge number of a limited augmented $5 \times 4$ bipartite graph, which has no generalized $C_4$ cycles,  only can reach $11$.  Note that this is the only place in this paper, where a degenerate 2-edge is actually involved.

\section{The $5 \times 5$ Case}

We now investigate the limited augmented Zarankiewicz number for $m = n = 5$.
The classical Zarankiewicz number is $z(5,5) = 12$.
We show that by introducing two carefully chosen $2$-edges,  %one nondegenerate and one row-degenerate,
we can obtain a limited augmented bipartite graph with total edge count $14$,
thereby establishing a lower bound for $z_L(5,5)$.

\begin{Thm} \label{thm:zL55}
$z_L(5,5) \ge 14$.
\end{Thm}

\begin{proof}
We construct an explicit limited augmented bipartite graph $G = (S, T, E_1 \cup E_2)$
with vertex sets $S = \{1,2,3,4,5\}$, $T = \{1,2,3,4,5\}$, and edge sets
\[
\begin{aligned}
E_1 = \{ &(1,1), (1,2), (1,3), \\
         &(2,1), (2,4), \\
         &(3,2), (3,4), (3,5), \\
         &(4,3), (4,5), \\
         &(5,1), (5,5) \},
\end{aligned}
\qquad
E_2 = \{ (1,4;5,2),\; (2,3;4,2) \}.
\]

Then $|E_1| = 12$ (the classical Zarankiewicz number) and $|E_2| = 2$,
so the total number of edges is $|E_1| + |E_2| = 14$.
{It is straightforward to verify that condition (S) holds and that Conditions 1 and 2 in Definition~\ref{d2.1} are not triggered. We now proceed to check Condition 3.}
We must check for each $2$-edge that no cell $(k,l)$ with the required disjointness
makes all five listed cells occupied.

\noindent\textbf{For $e_1 = (1,4;5,2)$ (nondegenerate).}
Here $i=1$, $j=4$, $p=5$, $q=2$.
We need $k \notin \{1,5\}$ and $l \notin \{4,2\}$,
i.e., $k \in \{2,3,4\}$, $l \in \{1,3,5\}$.
The five cells are:
\[
(k,l),\ (k,4),\ (k,2),\ (1,l),\ (5,l).
\]

We examine candidates:
\begin{itemize}
  \item $l = 1$: $(1,1) \in E_1$, $(5,1) \in E_1$.
        Need $(k,1)$ occupied for some $k$.
        $(2,1) \in E_1$, but $(2,4) \in E_1$ and $(2,2)$ not occupied $\to$ fails.
        $(3,1)$ not occupied, $(4,1)$ not occupied.
        So no violation.
  \item $l = 3$: $(1,3) \in E_1$, $(5,3)$ not occupied $\to$ fails.
  \item $l = 5$: $(1,5)$ not occupied $\to$ fails.
\end{itemize}
Thus no $(k,l)$ triggers Condition 3 for $e_1$.

\noindent\textbf{For $e_2 = (2,3;4,2)$ (nondegenerate).}
Here $i=2$, $j=3$, $p=4$, $q=2$.
We need $k \notin \{2,4\}$ and $l \notin \{3,2\}$,
i.e., $k \in \{1,3,5\}$, $l \in \{1,4,5\}$.
The five cells are:
\[
(k,l),\ (k,3),\ (k,2),\ (2,l),\ (4,l).
\]

We examine candidates:
\begin{itemize}
  \item $l = 1$: $(2,1) \in E_1$, $(4,1)$ not occupied $\to$ fails.
  \item $l = 4$: $(2,4) \in E_1$, $(4,4)$ not occupied $\to$ fails.
  \item $l = 5$: $(2,5)$ not occupied $\to$ fails.
\end{itemize}
Thus no $(k,l)$ triggers Condition 3 for $e_2$.

Since all three conditions are satisfied,
$G$ contains no generalized $C_4$-cycle.
Hence the associated doubly simple biquadratic form is irreducible, and
\[
\operatorname{sos}(P_G) = |E_1| + |E_2| = 12 + 2 = 14.
\]
Therefore $z_L(5,5) \ge 14$.
\end{proof}

\noindent\textbf{Remark.}
Unlike the $5 \times 4$ case, which required a degenerate 2-edge to achieve its lower bound, the
$5 \times 5$ lower bound is attained using only nondegenerate 2-edges.

\section{Conclusion and Open Problems}

This paper introduced the augmented Zarankiewicz number \(z_A(m,n)\) and the limited augmented Zarankiewicz number \(z_L(m,n)\) as combinatorial extensions of the classical Zarankiewicz number. Each such graph \(G\) corresponds to a doubly simple biquadratic form \(P_G\), and the main theoretical result (Theorem~2.4) establishes the inequality chain
\[
\operatorname{BSR}(m, n) \ge z_A(m, n) \ge z_L(m, n) \ge z(m, n),
\]
linking the maximum biquadratic SOS rank to these new graph parameters.

Our main combinatorial results determine the \emph{exact} values of the limited augmented Zarankiewicz number for all dimensions not exceeding four:
\[
\begin{aligned}
&z_L(m,2) = m + 1 \quad (m \ge 2), \\
&z_L(3,3) = 6, \\
&z_L(4,3) = 8, \\
&z_L(4,4) = 10.
\end{aligned}
\]
For parameters involving five rows, we have established new lower bounds that improve upon the classical Zarankiewicz numbers:
\[
\begin{aligned}
&z_L(5,3) \ge 9, \\
&z_L(5,4) \ge 12, \\
&z_L(5,5) \ge 14.
\end{aligned}
\]
\begin{table}[htbp]
	\centering
	\caption{Summary of lower bounds for $\operatorname{BSR}(m,n)$ of small \(m\) and \(n\). The values printed in bold denote cases where the lower bound is known to be tight.}
	\label{tab:small_booktabs}
	\setlength{\tabcolsep}{35pt}  % ???6pt????12pt????
	\begin{tabular}{c|ccccccccc}
		\toprule
		\(m \setminus n\) & 2 & 3 & 4 & 5 \\
		\midrule
		2  & \bf 3 &\bf 4 &\bf 5 & \bf 6 \\
		3  &\bf 4  & \bf 6  &  8 & 9  \\
		4  & \bf 5 &   8 &  10  & 12  \\
		5  & \bf 6  & 9   & 12  & 14  \\
		\bottomrule
	\end{tabular}
\end{table}

{We summarize lower bounds for $\operatorname{BSR}(m,n)$ of small \(m\) and \(n\) in Table~\ref{tab:small_booktabs}.}
These findings demonstrate that the augmented Zarankiewicz framework captures combinatorial obstructions to higher SOS rank that are invisible to the classical theory, thereby providing improved lower bounds for \(\operatorname{BSR}(m,n)\).   Note that our method is more efficient than the orthogonality method as $m$ and $n$ grow.   For example, to check the irreducibility of the example in the proof of Theorem 7.1 by the orthogonality method, one needs to check orthogonality relations.  The work grows
as $m$ and $n$ grow.

The work presented here lays the foundation for this new combinatorial approach. Several natural and challenging directions remain for future research.

\begin{enumerate}
    \item \textbf{Complete the $
5 \times n$ cases.} A key next step is to prove that the lower bounds obtained in this paper for the $5\times3$, $5\times4$, and $5\times5$ cases are sharp, i.e., that \(z_L(5,3)=9\), \(z_L(5,4)=12\), and \(z_L(5,5)=14\). Verifying these exact values requires a full case analysis over all non-isomorphic extremal graphs for the classical numbers \(z(5,3)=8\), \(z(5,4)=10\), and \(z(5,5)=12\). The classification of extremal \(C_4\)-free bipartite graphs for these parameters is known in the literature \cite{RS65, Gu69}, and the number of non-isomorphic extremal graphs is finite and manageable. The increased number of distinct frames makes this a more extensive combinatorial project, which will be the subject of a forthcoming work.

    \item \textbf{Asymptotic behaviour.} The classical Zarankiewicz number satisfies the K\H{o}v\'{a}ri--S\'{o}s--Tur\'{a}n bound \(z(m,n) = O(m n^{1/2} + n)\). A fundamental open problem is to determine the asymptotic growth rate of \(z_L(m,n)\). Does the introduction of 2-edges lead to a strictly larger asymptotic order, or does the same $O(m n^{1/2})$ upper bound persist?

    \item \textbf{Gap between \(z_L(m,n)\) and \(\operatorname{BSR}(m,n)\).} Our construction shows that \(\operatorname{BSR}(m,n)\ge z_L(m,n)\). If a gap exists, several issues can be investigated.   First, is $z_A(m, n) = z_L(m, n)$?  Second, should we relax the simpilicity condition to allow two 2-edges overlap?  Third, should we consider more general SOS representations that involve squares of bilinear forms with more than two terms?
\end{enumerate}

The problem investigated here lies at the intersection of algebraic geometry \cite{BPSV19, BSSV22}, extremal graph theory \cite{Bo04, CHM24, Gu69, KST54, RS65, Za51}, and hypergraph theory. Recent developments in hypergraph Zarankiewicz-type problems \cite{HHL26, JiangMa18} suggest that higher-order analogues of our construction may exist, where squares of multilinear forms correspond to edges in $k$-uniform hypergraphs. Such connections point to a rich interplay between SOS representations and extremal combinatorics that merits further exploration.

The results for dimensions \(m, n\le 5\) obtained in this paper represent a significant first step toward understanding this deeper connection. We hope that the concepts of \(z_A(m,n)\) and \(z_L(m,n)\) will provide a productive framework for future investigations into both the combinatorial and algebraic aspects of the biquadratic SOS rank problem.

	\bigskip
	
	\noindent\textbf{Acknowledgement}
	This work was partially supported by Research Center for Intelligent Operations Research, The Hong Kong Polytechnic University (4-ZZT8), the National Natural Science Foundation of China (Nos. 12471282 and 12131004),  and Jiangsu Provincial Scientific Research Center of Applied Mathematics (Grant No. BK20233002).
	
	\medskip
	
	\noindent\textbf{Data availability}
	No datasets were generated or analysed during the current study.
	
	\medskip
	
	\noindent\textbf{Conflict of interest} The authors declare no conflict of interest.

\end{document}